\documentclass[11pt]{article}
\usepackage{amssymb,amsthm,amsmath}

\newtheorem{thm}{Theorem}[section]
\newtheorem{lem}[thm]{Lemma}

\theoremstyle{definition}
\newtheorem{df}[thm]{Definition}

\numberwithin{equation}{section}

\renewcommand{\phi}{\varphi}

\newcommand{\N}{\mathbb{N}}
\newcommand{\Z}{\mathbb{Z}}

\title{An absorption theorem for \\
minimal AF equivalence relations on Cantor sets}
\author{Hiroki Matui 
\thanks{Supported in part by a grant 
from the Japan Society for the Promotion of Science} \\
Graduate School of Science \\
Chiba University \\
1-33 Yayoi-cho, Inage-ku, Chiba 263-8522, Japan}
\date{}

\begin{document}
\maketitle

\begin{abstract}
We prove that 
a `small' extension of a minimal AF equivalence relation 
on a Cantor set is orbit equivalent to the AF relation. 
By a `small' extension we mean an equivalence relation 
generated by the minimal AF equivalence relation and 
another AF equivalence relation 
which is defined on a closed thin subset. 
The result we obtain is 
a generalization of the main theorem in \cite{GMPS2}. 
It is needed 
for the study of orbit equivalence of minimal $\Z^d$-systems 
for $d>2$ \cite{GMPS3}, 
in a similar way 
as the result in \cite{GMPS2} was needed (and sufficient) 
for the study of minimal $\Z^2$-systems \cite{GMPS1}. 
\end{abstract}

\section{Introduction}

In the present paper we study equivalence relations on Cantor sets. 
By a Cantor set, 
we mean a compact, metrizable and totally disconnected space 
without isolated points. 
The topological orbit structure of countable group actions 
as homeomorphisms on Cantor sets has been studied by several authors 
\cite{GPS1}, \cite{GMPS1}. 
More precisely, minimal $\Z$-actions and $\Z^2$-actions on Cantor sets 
have been classified up to orbit equivalence. 
The strategy is to prove that 
the equivalence relation associated with the given minimal action 
is orbit equivalent to an AF relation (see Definition \ref{AF}). 
To prove this, we need a delicate `glueing' procedure, 
an essential part of which is done by the absorption theorem 
(\cite[Theorem 4.18]{GPS2}, \cite[Theorem 4.6]{GMPS2}). 
Indeed, the result in \cite{GMPS2} was sufficient 
for the study of orbit equivalence of minimal $\Z^2$-actions \cite{GMPS1}. 
The aim of this paper is 
to prove a stronger version of the absorption theorem, 
which is needed 
for the study of minimal $\Z^d$-actions for $d>2$ \cite{GMPS3}. 
We refer to \cite{GPS2} and \cite{GMPS2} 
as both background and reference for specific results 
that we shall need in the sequel. 

We will give a brief description of 
how a strengthening of the absorption theorem is needed 
in order to generalize the results for minimal $\Z^2$-actions 
to minimal $\Z^d$-actions. 
Let $\phi$ be a minimal free $\Z^d$-action on a Cantor set. 
For the associated equivalence relation $R_\phi$, 
we will construct an increasing sequence of subrelations 
$R_0\subset R_1\subset\dots\subset R_d=R_\phi$ 
so that $R_0$ is a minimal AF equivalence relation 
with the relative topology from $R_\phi$ 
and each $R_i$ is a `small' extension of $R_{i-1}$. 
Then, we apply inductively the absorption theorem to $R_{i-1}\subset R_i$ 
and show that each $R_i$ is orbit equivalent to an AF relation 
for $i=1,2,\dots,d$. 
In such a way, after $d$-times use of the absorption theorem, 
we can conclude that $R_d=R_\phi$ is orbit equivalent to an AF relation 
and thus complete the classification up to orbit equivalence. 
One of the problems in this argument is 
to describe the difference between $R_{i-1}$ and $R_i$. 
In the case of $d=2$, 
we could find another compact relation $K_i$ 
which is (locally) transverse to $R_{i-1}$ 
so that $R_i$ is generated by $R_{i-1}$ and $K_i$ (see \cite{GMPS1}). 
For $d>2$, however, we cannot find such a nice transverse relation, 
and so it is necessary to generalize the absorption theorem in \cite{GMPS2}. 
The new absorption theorem (Theorem \ref{main2}) in this paper 
does not need transverse relations 
and that is what is needed for the study of $\Z^d$-actions. 

We collect notation and terminology relevant to this paper. 
Let $X$ be a compact, metrizable and totally disconnected space 
and let $R\subset X\times X$ be an equivalence relation 
(we may call an equivalence relation just a relation). 
For a subset $A\subset X$, 
we set 
\[
R[A]=\{x\in X\mid
\text{there exists }y\in A\text{ such that }(x,y)\in R\}. 
\]
The set $R[A]$ is called the $R$-saturation of $A$. 
For $x\in X$, we denote $R[\{x\}]$ by $R[x]$ and 
call it the $R$-orbit of $x$. 
We deal with only an equivalence relation with countable orbits 
(i.e. $R[x]$ is at most countable for each $x\in X$). 
When $R[x]$ is dense in $X$ for each $x\in X$, 
we say that $R$ is minimal. 
For a subset $A\subset X$, 
we denote $R\cap(A\times A)$ by $R|A$ and call it the restriction. 
When $R$ and $S$ are relations on $X$, 
we let $R\vee S$ denote the equivalence relation on $X$ 
generated by $R$ and $S$. 

Suppose that $R$ is equipped with a topology 
in which $R$ is \'etale (\cite[Definition 2.1]{GPS2}). 
A closed subset $Y\subset X$ is called $R$-\'etale, 
if the restriction $R|Y=R\cap(Y\times Y)$ 
with the relative topology from $R$ is \'etale. 
A subset $Y\subset X$ is called $R$-thin, 
if $\mu(Y)$ is zero 
for any $R$-invariant probability measure $\mu$ on $X$. 

We collect several basic facts about \'etale equivalence relations. 
The reader should see \cite{GPS2} and \cite{GMPS2}. 
Let $R$ be an \'etale relation on a Cantor set $X$. 
If $O\subset X$ is open, then its $R$-saturation $R[O]$ is also open. 
If $R$ is compact, then the topology on $R$ coincides with 
the topology from the product topology of $X\times X$. 
If $R$ is compact and $O\subset X$ is clopen, then 
the $R$-saturation $R[O]$ is also clopen (and hence compact). 
One can easily show that 
a subrelation $S$ of $R$ is \'etale 
with respect to the relative topology from $R$ 
if and only if $S$ is an open subset of $R$. 
If $\mu(Y)=0$ for a Borel subset $Y$ of $X$ and 
an $R$-invariant probability measure $\mu$, then 
$\mu(R[Y])$ is also zero. 

\bigskip

The following is the definition of AF equivalence relations. 

\begin{df}[{\cite[Definition 3.7, 4.1]{GPS2}}]\label{AF}
An \'etale equivalence relation $R$ is called an AF relation, 
if there exists an increasing sequence $R_1\subset R_2\subset\dots$ 
of compact open subrelations of $R$ such that 
$R=\bigcup_{n\in\N}R_n$. 
An equivalence relation $R$ is said to be affable, 
if $R$ is orbit equivalent to an AF relation. 
\end{df}

We have to recall the notion of Bratteli diagrams. 
A Bratteli diagram $(V,E)$ consists of 
a vertex set $V$ and an edge set $E$, 
where $V$ and $E$ can be written 
as a countable disjoint union of non-empty finite sets: 
\[
V=V_0\cup V_1\cup V_2\cup\dots 
\text{ and }E=E_1\cup E_2\cup E_3\cup\dots
\]
with the following property: 
An edge $e$ in $E_n$ goes from a vertex in $V_{n-1}$ to one in $V_n$, 
which we denote by $s(e)$ and $r(e)$, respectively. 
We require that there are no sinks, 
i.e. $s^{-1}(v)\neq\emptyset$ for all $v\in V$. 
If $(V,E)$ has only one source $v_0\in V$
---which necessarily entails $V_0=\{v_0\}$---
we will call $(V,E)$ a standard Bratteli diagram. 

For a standard Bratteli diagram $(V,E)$, 
\[
X_{(V,E)}=\left\{(e_1,e_2,\dots)\in\prod_{n\in\N}E_n\mid
r(e_n)=s(e_{n+1})\text{ for all }n\in\N\right\}
\]
is called the infinite path space. 
Equipped with the relative topology from $\prod_{n\in\N}E_n$, 
$X_{(V,E)}$ is compact, metrizable and totally disconnected. 
For every $n\in\N$, let 
\[
R_n=\{(e,f)\in X_{(V,E)}\times X_{(V,E)}\mid
e_k=f_k\text{ for all }k>n\}, 
\]
where $e_k$ and $f_k$ denote the $k$-th edge of $e$ and $f$, 
respectively. 
Give $R_n$ the relative topology from $X_{(V,E)}\times X_{(V,E)}$. 
Then $R_n$ is a compact \'etale equivalence relation. 
Let 
\[
AF(V,E)=\bigcup_nR_n
\]
and give $AF(V,E)$ the inductive limit topology, 
so that $AF(V,E)$ is an AF equivalence relation. 

It is known that 
$AF(V,E)$ is the prototype of an AF relation. 
More precisely, for any AF equivalence relation $R$ 
on a compact, metrizable totally disconnected space $X$, 
there exists a standard Bratteli diagram $(V,E)$ such that 
$R$ is isomorphic to $AF(V,E)$ (\cite[Theorem 3.9]{GPS2}). 

We need the following lemma in the next section. 
We have been unable to find a suitable reference 
in the literature, and so we include a proof for completeness. 

\begin{lem}\label{proper}
Let $X$ be a compact metrizable totally disconnected space. 
Suppose $R$ and $S$ are compact \'etale equivalence relations on $X$. 
If $S$ is contained in $R$, 
then there exists a finite set $K$ and 
a continuous map $\mu:X\to K$ such that 
\[
S=\{(x,x')\in R\mid\mu(x)=\mu(x')\}. 
\]
\end{lem}
\begin{proof}
First, we note that $S$ is automatically open in $R$ 
(see the comment following Definition 3.7 in \cite{GPS2} for example). 

Let $Y$ be the quotient space of $X$ by the relation $S$. 
From Proposition 3.2 of \cite{GPS2} and its proof, 
we can see that $Y$ is compact, metrizable and totally disconnected. 
Let us denote the quotient map by $\pi$. 

For $f\in C(Y,\Z)$, we define 
\[
R_f=\{(x,x')\in R\mid f(\pi(x))=f(\pi(x'))\}. 
\]
It is easy to see that $R_f$ is a closed subset of $R$ and 
that $S$ is contained in $R_f$. 
If $(x,x')$ does not belong to $S$, then 
there exists $f\in C(Y,\Z)$ such that $f(\pi(x))\neq f(\pi(x'))$. 
Hence we have 
\[
S=\bigcap_{f\in C(Y,\Z)}R_f. 
\]
Since $S$ is open in $R$ and $R$ is compact, 
there exists a finite subset $A\subset C(Y,\Z)$ such that 
\[
S=\bigcap_{f\in A}R_f. 
\]
Put $K=\{(f(y))_{f\in A}\in\Z^A\mid y\in Y\}$ and 
define $\mu:X\to K$ by $\mu(x)=(f(\pi(x)))_{f\in A}$. 
It is easy to see that 
$K$ and $\mu$ have the desired properties. 
\end{proof}

\section{A splitting theorem}

Let $R$ be a minimal AF equivalence relation 
on a Cantor set $X$ and 
let $Y\subset X$ be a closed, $R$-\'etale and $R$-thin subset. 
By Theorem 3.11 of \cite{GPS2}, 
$R|Y=R\cap(Y\times Y)$ with the relative topology is 
an AF equivalence relation on $Y$. 
Suppose that we are given an equivalence relation $S$ on $Y$ 
and that $S$ is an open subset of $R|Y$. 
Note that $S$ in the relative topology from $R$ is also 
an AF equivalence relation on $Y$ by \cite[Proposition 3.12 (ii)]{GPS2}. 

We would like to prove the following theorem in this section. 

\begin{thm}\label{main1}
In the setting above, 
there exists an equivalence relation $R'$ on $X$ 
which satisfies the following. 
\begin{enumerate}
\item $R'$ is an open subset of $R$. 
\item $R'$ is minimal. 
\item $R'|Y$ is equal to $S$. 
\item $R'[Y]$ is equal to $R[Y]$. 
\item If $x\in X$ does not belong to $R[Y]$, then $R'[x]=R[x]$. 
\item Any $R'$-invariant probability measure on $X$ is $R$-invariant. 
\end{enumerate}
\end{thm}

The property (3) of the above theorem means that, 
for every $y\in Y$, 
its $R$-orbit $R[y]$ splits into several $R'$-orbits and 
$R'[y]\cap Y$ equals $S[y]$. 
But, the property (5) means that 
if $R[x]$ does not meet $Y$, then $R[x]$ does not split. 
Note that (4) and (5) imply $R=R'\vee(R|Y)$. 
\bigskip

At first, 
we need to represent the AF equivalence relation $R$ on $X$ 
by a Bratteli diagram. 
By Theorem 3.11 of \cite{GPS2}, 
there exists a standard Bratteli diagram $(V,E)$, 
a subdiagram $(W,F)$ (i.e. $W\subset V$, $F\subset E$) 
satisfying $r(F)\cup\{v_0\}=W$ 
and a homeomorphism $\pi:X\to X_{(V,E)}$ 
such that the following are satisfied. 
\begin{itemize}
\item $\pi\times\pi$ induces an isomorphism from $R$ to $AF(V,E)$. 
\item $\pi(Y)$ is equal to 
$\{(e_n)_n\in X_{(V,E)}\mid e_n\in F\text{ for all }n\in\N\}$. 
\end{itemize}
Note that $\pi|Y\times\pi|Y$ induces 
an isomorphism between $R|Y$ and $AF(W,F)$. 
To simplify notation, 
we identify $X_{(V,E)}$ with $X$ and omit $\pi$. 
We remark that $(V,E)$ is a simple Bratteli diagram, 
because $R$ is minimal. 
Moreover, $(W,F)$ is a thin subdiagram of $(V,E)$, 
because $Y$ is $R$-thin in $X$. 

Let $R_0$ be the trivial relation on $X$, that is, 
$R_0=\{(x,x)\mid x\in X\}$. 
For $n\in\N$, we define 
\[
R_n=\{(x,x')\in X\times X\mid x_k=x'_k\text{ for all }k>n\}, 
\]
where $x_k$ and $x'_k$ denote the $k$-th edge of 
infinite paths $x$ and $x'$, respectively. 
Notice that $R_0\subset R_1\subset R_2\subset\dots$ and 
$R=\bigcup_nR_n$. 

Since $S$ is an AF relation, 
there exists an increasing sequence of compact open subrelations 
$S_1\subset S_2\subset S_3\subset\dots$ in $S$ such that 
$S=\bigcup_mS_m$. 
For any $m\in\N$, $S_m$ is contained in $R|Y$ and 
$R|Y$ is a union of open subsets $R_n|Y$. 
It follows from the compactness of $S_m$ that 
there exists an increasing sequence $n_1<n_2<\dots$ 
such that $S_m\subset R_{n_m}|Y$ for all $m\in\N$. 
By telescoping $(V,E)$ to levels $0<n_1<n_2<\dots$, 
we may assume that $S_n\subset R_n|Y=R_n\cap(Y\times Y)$ for all $n\in\N$. 

Let $v_0$ be the unique vertex in $V_0$. 
For $v\in V_n$ and $w\in V_m$ with $0\leq n<m$, 
we denote the set of paths in $(V,E)$ from $v$ to $w$ by $E(v,w)$. 
Let $F(v,w)$ be the set of paths $(e_1,e_2,\dots,e_{m-n})$ 
in $E(v,w)$ such that $e_i\in F$ for all $i=1,2,\dots,m-n$. 

\begin{lem}
There exists an increasing sequence of non-negative integers 
$\{n(k)\}_{k=0}^\infty$ with $n(0)=0$ such that 
\[
\lvert F(v_0,w)\rvert\leq
\sum_{v\in V_{n(k-1)}}\lvert E(v,w)\setminus F(v,w)\rvert
\]
for all $w\in W_{n(k)}$ and $k\in\N$. 
\end{lem}
\begin{proof}
By Lemma 4.12 of \cite{GPS2}, we can find $n(1)\geq1$ such that 
$2\lvert F(v_0,w)\rvert\leq\lvert E(v_0,w)\rvert$ 
for all $w\in W_{n(1)}$, 
which means 
\[
\lvert F(v_0,w)\rvert\leq\lvert E(v_0,w)\setminus F(v_0,w)\rvert
\]
for all $w\in W_{n(1)}$. 

Put 
\[
L_n=\max_{v\in W_n}\lvert F(v_0,v)\rvert. 
\]
Let us find $n(2),n(3),n(4),\dots$ inductively. 
Suppose that $n(k-1)$ has been chosen. 
Since $Y$ is $R$-thin, by Lemma 4.12 of \cite{GPS2}, 
there exists $n(k)>n(k-1)$ such that 
\[
(L_{n(k-1)}+1)\lvert F(v,w)\rvert\leq\lvert E(v,w)\rvert
\]
for all $v\in W_{n(k-1)}$ and $w\in W_{n(k)}$. 
It follows that 
\begin{align*}
\lvert F(v_0,w)\rvert &\leq 
\sum_{v\in W_{n(k-1)}}L_{n(k-1)}\lvert F(v,w)\rvert \\
&\leq \sum_{v\in W_{n(k-1)}}\lvert E(v,w)\setminus F(v,w)\rvert \\
&\leq \sum_{v\in V_{n(k-1)}}\lvert E(v,w)\setminus F(v,w)\rvert 
\end{align*}
for any $w\in W_{n(k)}$. 
\end{proof}

From the lemma above, 
by telescoping $(V,E)$ to levels $0=n(0)<n(1)<n(2)<\dots$, 
we may assume that 
\[
\lvert F(v_0,w)\rvert\leq
\sum_{v\in V_{n-1}}\lvert E(v,w)\setminus F(v,w)\rvert
\qquad\text{for all }w\in W_n\text{ and }n\in\N. 
\]
Therefore, for $w\in W$, 
we can find a surjective map $\rho_w$ from 
$\{e\in E\setminus F\mid r(e)=w\}$ to $F(v_0,w)$. 

\begin{lem}\label{KlmdU}
There exist finite sets $K_n$, 
continuous maps $\lambda_n:X\to K_n$ and 
clopen subsets $U_n\subset X$ which satisfy the following. 
\begin{enumerate}
\item For every $n\in\N$, 
$S_n=\{(y,y')\in R_n\cap(Y\times Y)\mid
\lambda_n(y)=\lambda_n(y')\}$. 
\item For every $n\in\N$, 
$Y$ is contained in $U_n$. 
\item For every $n\in\N$, 
$\displaystyle\bigcap_{m\geq n}R_m[U_m]=R_n[Y]$. 
\item For every $n\in\N\setminus\{1\}$, 
if $(x,x')\in R_{n-1}$ and $\lambda_{n-1}(x)=\lambda_{n-1}(x')$, 
then $\lambda_n(x)=\lambda_n(x')$. 
\item For every $n\in\N$, 
if $x,x'\notin R_n[U_n]$, then 
$\lambda_n(x)=\lambda_n(x')$. 
\item For every $n\in\N$ and $y\in Y$, 
there exists $x\in R_n[y]$ such that 
if $(x,x')\in R_{n-1}$, then $\lambda_n(x')=\lambda_n(y)$. 
\item For every $n\in\N\setminus\{1\}$ and $y\in U_n$, we have 
\[
\min_{v\in V_{n-1}}\lvert E(v_0,v)\rvert\times
\lvert\{x\in R_n[y]\cap U_n\mid\lambda_n(x)=\lambda_n(y)\}\rvert
\leq\lvert\{x\in R_n[y]\mid\lambda_n(x)=\lambda_n(y)\}\rvert. 
\]
\item For every $n\in\N$ and $x\in R_n[Y]$, 
there exists $y\in Y$ such that $(x,y)\in R_n$ and 
$\lambda_n(x)=\lambda_n(y)$. 
\end{enumerate}
\end{lem}
\begin{proof}
Since $S_n$ is contained in $R_n|Y$, 
by applying Lemma \ref{proper}, 
we get a finite set $K_n$ and a continuous map $\mu_n:Y\to K_n$ such that 
\begin{equation}
S_n=\{(y,y')\in R_n|Y\mid\mu_n(y)=\mu_n(y')\}. 
\label{dog}
\end{equation}
For $k\in\N$, we define 
\[
Y_k=\{(x_n)_n\in X\mid x_n\in F\text{ for all }n=1,2,\dots,k\}. 
\]
The clopen sets $Y_k$'s form a decreasing sequence and 
$\bigcap_kY_k=Y$. 
For $w\in W$, let $\rho_w$ be a surjective map 
from $\{e\in E\setminus F\mid r(e)=w\}$ to $F(v_0,w)$ as above. 

First of all, let us find $U_1$ and $\lambda_1:X\to K_1$. 
Put $U_1=Y_2$. Then (2) for $n=1$ is clear. 
Let $\tilde{\mu}_1:U_1\to K_1$ be an arbitrary continuous extension 
of $\mu_1:Y\to K_1$. 
For $x\in U_1$, we define $\lambda_1(x)=\tilde{\mu}_1(x)$ 
This, together with \eqref{dog}, implies (1) for $n=1$. 
On $X\setminus R_1[U_1]$, we fix an element of $K_1$ and 
let $\lambda_1$ be the constant map to this element, 
so that (5) is satisfied. 
Suppose that $x$ is in $R_1[U_1]\setminus U_1$. 
Let $x_k$ denote the $k$-th edge of the infinite path $x\in X$. 
It is easy to see $x_1\notin F$ and $r(x_1)\in W$. 
Since $x_2\in F$, 
\[
\tilde{x}=(\rho_{r(x_1)}(x_1),x_2,x_3,\dots)\in X
\]
belongs to $U_1$. 
Hence we can define $\lambda_1(x)=\lambda_1(\tilde{x})$. 
One observes that $\lambda_1:X\to K_1$ is continuous. 
To check (8), let $x\in R_1[Y]\setminus Y$. 
From $x_1\notin F$, $r(x_1)\in W$ and $x_2\in F$, 
we can see that $x$ belongs to $R_1[U_1]\setminus U_1$. 
Obviously, $\tilde{x}=(\rho_{r(x_1)}(x_1),x_2,x_3,\dots)$ is in $Y$, 
and so (8) for $n=1$ follows. 

We would like to construct $U_n$ and $\lambda_n:X\to K_n$ 
inductively. 
Let us assume that $U_{n-1}$ and $\lambda_{n-1}$ have been fixed. 
Let $\tilde{\mu}_n:Y_{n+1}\to K_n$ be 
an arbitrary continuous extension of $\mu_n:Y\to K_n$. 
We claim that there exists $k>n$ such that 
if $x,x'\in Y_k$ satisfies 
$(x,x')\in R_{n-1}$ and $\lambda_{n-1}(x)=\lambda_{n-1}(x')$, then 
$\tilde{\mu}_n(x)=\tilde{\mu}_n(x')$. 
Otherwise, for each $k>n$, we would have $x(k),x'(k)\in Y_k$ with 
$(x(k),x'(k))\in R_{n-1}$, $\lambda_{n-1}(x(k))=\lambda_{n-1}(x'(k))$ and 
$\tilde{\mu}_n(x(k))\neq\tilde{\mu}_n(x'(k))$. 
We may assume that 
two sequences $x(k),x'(k)$ converge to $y,y'\in Y$, respectively, 
because $X$ is compact and $\bigcap Y_k=Y$. 
By compactness of $R_{n-1}$, we also have $(y,y')\in R_{n-1}$. 
Combining this with $\lambda_{n-1}(y)=\lambda_{n-1}(y')$, 
by (1) for $n-1$, we get $(y,y')\in S_{n-1}$. 
On the other hand, by \eqref{dog} and $\mu_n(y)\neq\mu_n(y')$, 
$(y,y')$ does not belong to $S_n$, 
which contradicts $S_{n-1}\subset S_n$. 
Hence we can find $k>n$ 
which has the desired property. 
We put $U_n=Y_k$, so that 
\begin{equation}
(x,x')\in R_{n-1}|U_n\text{ and }
\lambda_{n-1}(x)=\lambda_{n-1}(x')
\quad\Rightarrow\quad\tilde{\mu}_n(x)=\tilde{\mu}_n(x'). 
\label{cat}
\end{equation}
Notice that $Y$ is contained in $U_n$ and 
$U_n$ is contained in $Y_{n+1}$. 

Next, we would like to define a continuous map $\lambda_n:X\to K_n$.
Fix an element $\kappa_0\in K_n$. 
Let $x\in R_{n-1}[U_n]$. 
If there exists $x'\in U_n$ such that 
$(x,x')\in R_{n-1}$ and $\lambda_{n-1}(x)=\lambda_{n-1}(x')$, 
then we define $\lambda_n(x)=\tilde{\mu}_n(x')$. 
This is well-defined because of \eqref{cat}. 
If there does not exist such $x'\in U_n$, 
then we define $\lambda_n(x)=\kappa_0$. 
Notice that this definition implies (1) and (4) for $x,x'\in R_{n-1}[U_n]$. 
For $x\notin R_n[U_n]$, we define $\lambda_n(x)=\kappa_0$, 
so that (5) is satisfied. 
Suppose that $x$ is in $R_n[U_n]\setminus R_{n-1}[U_n]$. 
Let $x_k\in E$ denote the $k$-th edge of $x$. 
From $x\notin R_{n-1}[U_n]$, we can see that $x_n\in E\setminus F$. 
Since $x$ is in $R_n[U_n]$, we also get $r(x_n)\in W$. 
By definition of $\rho_{r(x_n)}$, 
$\rho_{r(x_n)}(x_n)$ is in $F(v_0,r(x_n))$. 
It follows that 
\[
\tilde{x}=(\rho_{r(x_n)}(x_n),x_{n+1},x_{n+2},\dots)
\]
belongs to $U_n$. 
Therefore we can define $\lambda_n(x)=\lambda_n(\tilde{x})$. 
We remark that, by definition, 
if $x,x'\in R_n[U_n]\setminus R_{n-1}[U_n]$ and $(x,x')\in R_{n-1}$, 
then $x_n=x'_n$, and hence $\rho_{r(x_n)}(x_n)=\rho_{r(x'_n)}(x'_n)$. 
Therefore $\lambda_n(x)=\lambda_n(x')$. 
Thus (4) for $x,x'\in R_n[U_n]\setminus R_{n-1}[U_n]$ is satisfied. 

Let us check (6). Take $y\in Y$. 
By the surjectivity of $\rho_{r(y_n)}$, 
there exists $e\in E\setminus F$ such that $r(e)=r(y_n)$ and 
\[
\rho_{r(y_n)}(e)=(y_1,y_2,\dots,y_n). 
\]
Take an infinite path $x\in X$ such that 
$x_n=e$ and $x_k=y_k$ for all $k>n$. 
It is easy to see that $x$ has the desired property. 

We next verify (7). Take $y\in U_n$. 
Since $U_n$ is contained in $Y_{n+1}$, 
by the same argument as above, 
we can choose $e\in E\setminus F$ such that $r(e)=r(y_n)$ and 
\[
\rho_{r(y_n)}(e)=(y_1,y_2,\dots,y_n). 
\]
Put 
\[
P_y=\{x\in X\mid x_n=e, \ x_k=y_k\text{ for all }k>n\}. 
\]
Notice that $\lvert P_y\rvert$ equals $\lvert E(v_0,s(e))\rvert$. 
It is clear that $(x,y)$ belongs to $R_n$ for every $x\in P_y$. 
From the definition of $\lambda_n$, 
we have $\lambda_n(x)=\lambda_n(y)$ for every $x\in P_y$. 
It is also clear that $x\notin U_n$ for any $x\in P_y$, 
because $e$ is in $E\setminus F$. 
Finally, if $y,y'\in U_n$ are distinct, then $P_y$ does not meet $P_{y'}$. 
This completes the proof of (7). 

Let us consider (8). Take $x\in R_n[Y]$. 
If $x$ is in $R_{n-1}[Y]$, then by the induction hypothesis 
there exists $y\in Y$ such that 
$(x,y)\in R_{n-1}$ and $\lambda_{n-1}(x)=\lambda_{n-1}(y)$. 
It follows from (4) that $\lambda_n(x)$ is equal to $\lambda_n(y)$. 
Suppose $x\in R_n[Y]\setminus R_{n-1}[Y]$, 
which means $x_n\in E\setminus F$ and $x_k\in F$ for all $k>n$. 
Thus, $x\in R_n[U_n]\setminus R_{n-1}[U_n]$. 
As before, 
we put $\tilde{x}=(\rho_{r(x_n)}(x_n),x_{n+1},x_{n+2},\dots)$. 
Then, $\tilde{x}$ belongs to $Y$ and 
$(x,\tilde{x})\in R_n$, $\lambda_n(x)=\lambda_n(\tilde{x})$. 

In this way, we can find $U_n$ and $\lambda_n:X\to K_n$ 
for every $n\in\N$. 
Finally, let us check (3). 
Since $U_m$ contains $Y$, 
$\bigcap_{m\geq n}R_m[U_m]\supset R_n[Y]$ is clear. 
By the construction of $U_m$, for every $m\in\N$, 
\[
R_m[U_m]\subset\{x\in X\mid x_{m+1}\in F\}. 
\]
As an immediate consequence, we have 
\[
\bigcap_{m\geq n}R_m[U_m]
\subset\{x\in X\mid x_m\in F\text{ for all }m>n\}
=R_n[Y]. 
\]
\end{proof}

Now we are ready to prove Theorem \ref{main1}. 

\begin{proof}[Proof of Theorem \ref{main1}]
Let $K_n$, $\lambda_n:X\to K_n$ and $U_n$ be as in the lemma above. 
Define 
\[
R_n'=\{(x,x')\in R_n\mid\lambda_n(x)=\lambda_n(x')\}
\]
for every $n\in\N$. 
It is clear that $R_n'$ is an open subset of $R_n$. 
Moreover, by (4) of Lemma \ref{KlmdU}, 
$R_{n-1}'$ is contained in $R_n'$. 
Put $R'=\bigcup R_n'$. 
Evidently $R'$ is an equivalence relation and 
an open subset of $R$. 
By (1) of Lemma \ref{KlmdU}, we have $R'|Y=S$. 

Let us show $R'[Y]=R[Y]$. 
Take $x\in R[Y]$. 
There exists $n\in\N$ such that $x\in R_n[Y]$. 
By (8) of Lemma \ref{KlmdU}, 
there exists $y\in Y$ such that 
$(x,y)\in R_n$ and $\lambda_n(x)=\lambda_n(y)$. 
Hence we get $(x,y)\in R_n'$, 
which means that $x$ is in $R'[Y]$. 

We would like to show condition (5) of Theorem \ref{main1}. 
Suppose that $x$ is not in $R[Y]$. 
In order to prove $R'[x]=R[x]$, take $x'\in R[x]$. 
We can find $n\in\N$ such that $(x,x')\in R_n$. 
By (3) of Lemma \ref{KlmdU}, 
there exists $m\geq n$ such that $R_m[U_m]$ does not contain $x$. 
Also, clearly $x'\notin R_m[U_m]$. 
It follows from (5) of Lemma \ref{KlmdU} 
that $\lambda_m(x)$ is equal to $\lambda_m(x')$. 
By definition of $R_m'$, we get $(x,x')\in R_m'$. 
Therefore $R'[x]=R[x]$. 

We now consider the minimality of $R'$. 
Take $x\in X$. 
We must show that $R'[x]$ is dense in $X$. 
If $x$ is not in $R[Y]$, as shown in the last paragraph, 
$R'[x]$ is equal to $R[x]$. 
Since $R$ is minimal, $R[x]$ is dense in $X$. 
Hence we may assume that $x$ is in $R[Y]$. 
As shown above, $R[Y]=R'[Y]$. 
It follows that we can find $y\in Y$ such that $(x,y)\in R'$. 
Take a non-empty open subset $O\subset X$ arbitrarily. 
The minimality of $R$ implies $R[O]=X$. 
Since $X$ is compact and $R[O]=\bigcup R_n[O]$, 
we can find $n\in\N$ such that $R_n[O]=X$. 
By (6) of Lemma \ref{KlmdU}, 
there exists $z\in R_{n+1}[y]$ such that 
$R_n[z]\subset R_{n+1}'[y]$. 
From $(x,y)\in R'$, we have $R_n[z]\subset R'[x]$. 
Combining this with $R_n[O]=X$, 
we can conclude that $R'[x]$ meets $O$, 
which implies $R'[x]$ is dense in $X$. 

It remains for us to show the last condition. 
To do that, we would like to show that $Y$ is $R'$-thin. 
From (7) of Lemma \ref{KlmdU}, for every $y\in U_n$, we have 
\[
\min_{v\in V_{n-1}}\lvert E(v_0,v)\rvert
\times\lvert R_n'[y]\cap U_n\rvert\leq\rvert R_n'[y]\rvert. 
\]
Notice that $R_n'$ is a compact relation. 
It follows that 
\[
\mu(U_n)\leq
\left(\min_{v\in V_{n-1}}\lvert E(v_0,v)\rvert\right)^{-1}
\]
for every $R'$-invariant probability measure $\mu$. 
The right-hand side converges to zero, 
because $R$ is minimal. 
Since $U_n$ contains $Y$, we get $\mu(Y)=0$. 

Let us show that 
any $R'$-invariant probability measure on $X$ is $R$-invariant. 
Let $\mu$ be an $R'$-invariant probability measure and 
let $\gamma:O_1\to O_2$ be a homeomorphism 
between clopen subsets $O_1,O_2\subset X$ such that 
$(x,\gamma(x))\in R$ for every $x\in O_1$, 
i.e. $\gamma$ is a graph in $R$. 
It suffices to show $\mu(O_1)=\mu(O_2)$. 
Since $Y$ is $R'$-thin and $R'[Y]=R[Y]$, 
we have $\mu(O_1)=\mu(O_1\setminus R[Y])$ 
and $\mu(O_2)=\mu(O_2\setminus R[Y])$. 
Clearly $\gamma(O_1\setminus R[Y])=O_2\setminus R[Y]$ and 
$(x,\gamma(x))\in R'$ for any $x\in O_1\setminus R[Y]$. 
Hence we get $\mu(O_1\setminus R[Y])=\mu(O_2\setminus R[Y])$, 
and so $\mu(O_1)$ is equal to $\mu(O_2)$. 
\end{proof}

\section{An absorption theorem}

In this section, by using Theorem \ref{main1}, 
we would like to prove the main theorem. 
We begin with a lemma. 

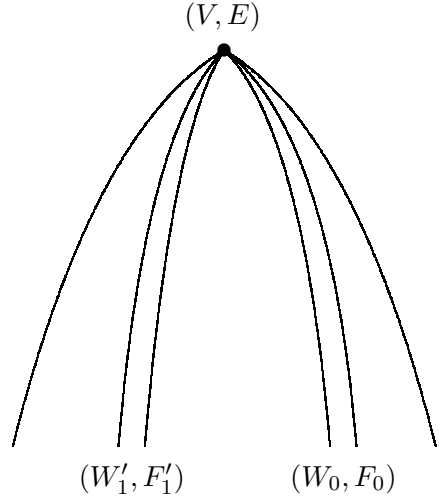
\begin{picture}(250,220)
\put(150,180){\circle*{5}}
\qbezier(70,30)(100,150)(150,180)
\qbezier(230,30)(200,150)(150,180)
\qbezier(110,30)(120,150)(150,180)
\qbezier(120,30)(130,150)(150,180)
\qbezier(190,30)(180,150)(150,180)
\qbezier(200,30)(190,150)(150,180)
\put(135,190){$(V,E)$}
\put(95,15){$(W_1',F_1')$}
\put(175,15){$(W_0,F_0)$}
\end{picture}

\begin{lem}\label{copy}
Let $R\subset X\times X$ be a minimal AF equivalence relation 
on a Cantor set $X$ and 
let $Y\subset X$ be a closed, $R$-\'etale and $R$-thin subset. 
Let $Z$ be a compact metrizable totally disconnected space 
and let $Q\subset Z\times Z$ be an AF equivalence relation on $Z$. 
Then, there exists a continuous map $\pi:Z\to X$ such that 
the following are satisfied. 
\begin{enumerate}
\item $\pi$ is a homeomorphism from $Z$ to $\pi(Z)$. 
\item $\pi(Z)$ is a closed, $R$-\'etale and $R$-thin subset. 
\item $\pi(Z)$ does not meet $R[Y]$. 
\item $\pi\times\pi$ gives a homeomorphism 
from $Q$ to $R\cap(\pi(Z)\times \pi(Z))$. 
\end{enumerate}
\end{lem}
\begin{proof}
As in the last section, we may assume that 
there exist a simple standard Bratteli diagram $(V,E)$ and 
its thin subdiagram $(W_0,F_0)$ such that 
the AF equivalence relation $R$ on $X$ is represented by $(V,E)$ 
and $R|Y$ corresponds to $(W_0,F_0)$. 
Similarly, by \cite[Theorem 3.9]{GPS2}, we may assume that 
$Q\subset Z\times Z$ is represented 
by another standard Bratteli diagram $(W_1,F_1)$. 

We now transform the Bratteli diagram $(V,E)$ 
by a succession of telescopings and microscopings 
so that the resulting diagram, which we again denote by $(V,E)$, 
can be described as follows (see also the figure). 
There are two disjoint thin subdiagrams of $(V,E)$. 
One is the subdiagram which is transformed from $(W_0,F_0)$ above, 
and we retain the notation for it. 
The other thin subdiagram is a replica of $(W_1,F_1)$, 
and we denote it by $(W'_1,F'_1)$. 

Let $\pi$ denote the canonical homeomorphism 
from the infinite path space on $(W_1,F_1)$, 
which is identified with $Z$, 
to the infinite path space on $(W'_1,F'_1)$, 
which is identified with a closed thin subset of $X$. 
Since $(W_0,F_0)$ and $(W'_1,F'_1)$ are disjoint, 
$\pi(Z)$ does not meet $R[Y]$. 
The other properties can be verified easily. 
\end{proof}

We are now ready to give a proof of the main result. 
For \'etale equivalence relations $Q$ and $R$, 
we say that $Q$ is an \'etale extension of $R$, 
if $Q$ contains $R$ and 
the inclusion map from $R$ to $Q$ is continuous. 

\begin{thm}\label{main2}
Let $R\subset X\times X$ be a minimal AF equivalence relation 
on a Cantor set $X$ and 
let $Y\subset X$ be a closed, $R$-\'etale and $R$-thin subset. 
Suppose that 
an AF equivalence relation $Q\subset Y\times Y$ is 
an \'etale extension of $R|Y$. 
Then we can find a homeomorphism $h:X\to X$ such that 
the following are satisfied. 
\begin{enumerate}
\item $h\times h(R\vee Q)=R$, where $R\vee Q$ is 
the equivalence relation generated by $R$ and $Q$. 
\item $h(Y)$ is a closed, $R$-\'etale and $R$-thin subset. 
\item $h|Y\times h|Y$ gives a homeomorphism 
from $Q$ to $R|h(Y)$. 
\end{enumerate}
In particular, $R\vee Q$ is affable. 
\end{thm}
\begin{proof}
The proof idea is the same as in the proof of the absorption theorem 
\cite[Theorem 4.6]{GMPS2}, 
namely constructing countable disjoint replicas of $R|Y$, respectively $Q$, 
inside a ``big'' equivalence relation, 
and use the extension result \cite[Lemma 4.15]{GPS2}. 

Let $Z=(Y\times\N)\cup\{\infty\}$ be 
the one-point compactification of $Y\times\N$. 
Set 
\[
\widetilde{Q}=\{((y,n),(y',n))\in Z\times Z\mid
(y,y')\in Q,n\in\N\}\cup\{(\infty,\infty)\}. 
\]
Since $Q$ is an AF relation, 
there exists an increasing sequence of 
compact open subrelations $Q_n\subset Q$ such that $Q=\bigcup_{n\in\N}Q_n$. 
For every $n\in\N$, we put 
\[
\widetilde{Q}_n=\{((y,k),(y',k))\in Z\times Z\mid
(y,y')\in Q_n,k=1,2,\dots,n\}\cup\{(z,z)\mid z\in Z\}. 
\]
It is not so hard to see that 
$\widetilde{Q}_n$ is a compact \'etale relation on $Z$ 
with the relative topology from $Z\times Z$. 
In addition, we have $\widetilde{Q}_n\subset\widetilde{Q}_{n+1}$ and 
$\widetilde{Q}=\bigcup_n\widetilde{Q}_n$. 
It follows that $\widetilde{Q}$ is an AF equivalence relation 
with the inductive limit topology. 
By Lemma \ref{copy}, 
there exists a continuous map $\pi:Z\to X$ such that 
the following properties are satisfied. 
\begin{itemize}
\item $\pi$ is a homeomorphism from $Z$ to $\pi(Z)$. 
\item $\pi(Z)$ is a closed, $R$-\'etale and $R$-thin subset. 
\item $\pi(Z)$ does not meet $R[Y]$. 
\item $\pi\times\pi$ is a homeomorphism 
from $\widetilde{Q}$ to $R|\pi(Z)=R\cap(\pi(Z)\times\pi(Z))$. 
\end{itemize}
From the second and third conditions, 
it follows that 
$Y\cup\pi(Z)$ is also $R$-\'etale and $R$-thin. 

We define an equivalence relation $S$ on $Z$ by 
\[
S=\{((y,n),(y',n))\in\widetilde{Q}\mid (y,y')\in R\}
\cup\{(\infty,\infty)\}. 
\]
It is a routine matter to verify that 
$S$ is an open subrelation of $\widetilde{Q}$. 
Therefore $\pi\times\pi(S)$ is an open subrelation of 
$\pi\times\pi(\widetilde{Q})=R|\pi(Z)$. 
By Theorem \ref{main1}, 
there exists a minimal open subrelation $R'\subset R$ such that 
the following properties are satisfied. 
\begin{itemize}
\item $R'|\pi(Z)=\pi\times\pi(S)$. 
\item $R'[\pi(Z)]=R[\pi(Z)]$. 
\item If $x$ is not in $R[\pi(Z)]$, then $R'[x]=R[x]$. 
In particular, $R'|Y=R|Y$. 
\item Any $R'$-invariant probability measure on $X$ is $R$-invariant. 
\end{itemize}
Evidently $Y\cup\pi(Z)$ is $R'$-\'etale and $R'$-thin, 
and we have 
\[
R=R'\vee(R|\pi(Z))
=R'\vee(\pi\times\pi(\widetilde{Q}))
\]
and 
\[
R\vee Q=R'\vee Q\vee(\pi\times\pi(\widetilde{Q})). 
\]
It is also easy to see 
\begin{align*}
R'|(Y\cup\pi(Z))
&=(R'|Y)\cup(R'|\pi(Z)) \\
&=(R|Y)\cup(R'|\pi(Z)) \\
&\cong (R|Y)\cup S \\
&\cong S, 
\end{align*}
where the last homeomorphism is obtained by an obvious shift map 
sending $n$ to $n+1$, cf. definition of $S$. 
We define a homeomorphism $h:Y\cup\pi(Z)\to\pi(Z)$ 
by $h(y)=\pi(y,1)$ for $y\in Y$, $h(\pi(y,n))=\pi(y,n+1)$ for $(y,n)\in Z$ 
and $h(\pi(\infty))=\pi(\infty)$. 
Then 
\[
h\times h:R'|(Y\cup\pi(Z))
\to R'|\pi(Z)
\]
is a homeomorphism. 
Note also that $h\times h$ implements an isomorphism between 
$Q\vee(\pi\times\pi(\widetilde{Q}))$ 
(which is a relation on $Y\cup\pi(Z)$) and 
$\pi\times\pi(\widetilde{Q})$ (which is a relation on $\pi(Z)$). 
This is an immediate consequence of 
the definition of $\widetilde{Q}$ and $\pi$. 
By \cite[Lemma 4.15]{GPS2}, 
$h$ extends to a homeomorphism $\tilde{h}:X\to X$ 
such that $\tilde{h}\times\tilde{h}(R')=R'$. 
It is clear that $\tilde{h}\times\tilde{h}(R\vee Q)$ equals $R$. 
Besides, $h(Y)=\pi(Y\times\{1\})$ is $R$-\'etale and $R$-thin. 
We can also check that 
$\tilde{h}\times\tilde{h}$ induces a homeomorphism 
from $Q\subset Y\times Y$ to $R|\tilde{h}(Y)$,
which completes the proof. 
\end{proof}

\bigskip

\noindent\textbf{Acknowledgement.} \\
The author is grateful to Christian Skau for many helpful comments.

\end{document}